\documentclass[11pt]{amsart}

\usepackage[a4paper,hmargin=3.5cm,vmargin=4cm]{geometry}
\usepackage{amsfonts,amssymb,amscd,amstext}
\usepackage{graphicx}
\usepackage[dvips]{epsfig}

\usepackage{fancyhdr}
\usepackage{enumerate}
\usepackage{titlesec}
\usepackage{mathrsfs}

\titleformat{\section}
{\filcenter\bfseries\large} {\thesection{.}}{0.2cm}{}
\titleformat{\subsection}[runin]
{\bfseries} {\thesubsection{.}}{0.15cm}{}[.]
\titleformat{\subsubsection}[runin]
{\em}{\thesubsubsection{.}}{0.15cm}{}[.]

\usepackage[up,bf]{caption}

\newtheorem{theorem}{Theorem}

\newtheorem{lemma}{Lemma}
\newtheorem{corollary}{Corollary}

\theoremstyle{definition}
\newtheorem{definition}{Definition}
\newtheorem{remark}{Remark}
\newtheorem{question}{Question}

\numberwithin{equation}{section}
\numberwithin{figure}{section}

\def\Acal{\mathcal{A}}
\def\Bcal{\mathcal{B}}
\def\Ccal{\mathcal{C}}
\def\Rcal{\mathcal{R}}
\def\Mcal{\mathcal{M}}
\def\Ncal{\mathcal{N}}
\def\Ucal{\mathcal{U}}

\def\Ascr{\mathscr{A}}
\def\Iscr{\mathscr{I}}
\def\Uscr{\mathscr{U}}
\def\Bscr{\mathscr{B}}
\def\c{\mathbb{C}}
\def\z{\mathbb{Z}}
\def\d{\mathbb{D}}
\def\b{\mathbb{B}}
\def\r{\mathbb{R}}
\def\n{\mathbb{N}}

\def\z{\mathbb{Z}}
\def\jgot{\mathfrak{j}}
\def\igot{\mathfrak{i}}

\def\dist{\mathrm{dist}}
\def\span{\mathrm{span}}
\def\length{\mathrm{length}}

\usepackage{color}

\begin{document}

\fancyhead[CO]{Every bordered Riemann surface is a complete proper curve in a ball} 
\fancyhead[CE]{A.\ Alarc\'{o}n and F.\ Forstneri\v c}
\fancyhead[RO,LE]{\thepage}

\thispagestyle{empty}

\vspace*{1cm}
\begin{center}
{\bf\LARGE Every bordered Riemann surface is a complete proper curve in a ball}

\vspace*{0.5cm}

{\large\bf Antonio Alarc\'{o}n and Franc Forstneri\v c}
\end{center}

\footnote[0]{\vspace*{-0.4cm}

\noindent A.\ Alarc\'{o}n

\noindent Departamento de Geometr\'{\i}a y Topolog\'{\i}a, Universidad de Granada, E-18071 Granada, Spain.

\noindent e-mail: {\tt alarcon@ugr.es}

\vspace*{0.1cm}

\noindent F.\ Forstneri\v c

\noindent Faculty of Mathematics and Physics, University of Ljubljana, and Institute
of Mathematics, Physics and Mechanics, Jadranska 19, 1000 Ljubljana, Slovenia.

\noindent e-mail: {\tt franc.forstneric@fmf.uni-lj.si}
}

\vspace*{1cm}

\begin{quote}
{\small
\noindent {\bf Abstract}\hspace*{0.1cm} We prove that every bordered Riemann surface admits a complete proper holomorphic immersion into a ball of $\c^2$, and a complete proper holomorphic embedding into a ball of $\c^3$.

\vspace*{0.1cm}

\noindent{\bf Keywords}\hspace*{0.1cm} Riemann surfaces, complex curves, complete holomorphic immersions.

\vspace*{0.1cm}

\noindent{\bf MSC (2010):}\hspace*{0.1cm} 32B15, 32H02, 14H50, 53C42.}
\end{quote}


\section{Introduction}\label{sec:intro}
The question of whether there exist complete bounded complex submanifolds of a Euclidean space $\c^n$ for $n>1$ was raised by Yang \cite{Yang1,Yang2}. (An immersed submanifold $f\colon M\to \c^n$ is said to be {\em complete} if the induced Riemannian metric $f^* ds^2$ on $M$, obtained by pulling back the Euclidean metric $ds^2$ on $\c^n$ by the immersion, is a complete metric on $M$. A submanifold is bounded if it is contained in a relatively compact subset of the ambient space.) The first important result in this direction was obtained by Jones \cite{Jones} who constructed a holomorphic immersion of the unit complex disc $\d$ into $\c^2$, and an embedding into $\c^3$, with bounded image and complete induced metric. His method is strongly complex analytic and is based on the BMO duality theorem. 

The closely related question on the existence of complete bounded minimal surfaces in $\r^3$ was a classical problem in the theory of minimal surfaces, known as the {\em Calabi-Yau problem}. The first affirmative answer was given by Nadirashvili \cite{Nad}; his method uses Runge's theorem for holomorphic functions on planar labyrinths of compact sets.

No new examples of complete bounded complex submanifolds in $\c^n$ were discovered until recently when Mart\'{i}n et al. \cite{MUY} extended Jones' result to complete bounded complex curves in $\c^2$ with arbitrary finite genus and finitely many ends. Their technique is completely different from the one of Jones and is much more geometric; they used the existence of a simply connected complete bounded null holomorphic curve in $\c^3$ \cite{AL} and modified a technique developed by F.\ J.\ L\'{o}pez \cite{Lopez} for constructing complete minimal surfaces in $\r^3$. 

Recently Alarc\'{o}n and L\'{o}pez \cite{AL} constructed complete bounded holo\-mor\-phic curves in $\c^2$ and null curves in $\c^3$ of arbitrary topological type. Their method also comes from the theory of minimal surfaces in $\r^3$ and null curves in $\c^3$, and it strongly relies on Runge's and Mergelyan's approximation theorems. Their examples have the extra feature that they can be made proper in any convex domain of $\c^2$; in particular, in the unit ball $\b=\{z\in\c^2\colon |z|<1\}$. In the same spirit, Ferrer et al. \cite{FMM} constructed complete minimal surfaces with arbitrary topology and properly immersed in any convex domain of $\r^3$.

Despite their flexibility, none of these methods enables one to control the complex structure on the curve, except of course in the simply connected case when any such curve is biholomorphic to the unit disc. 

The aim of this paper is to construct complete bounded complex curves that are normalized by any given bordered Riemann surface. Our main result is the following.

\begin{theorem}\label{th:intro}
Every bordered Riemann surface admits a complete proper holomorphic immersion to the unit ball of $\c^2$, and a complete proper holomorphic embedding to the unit ball of $\c^3$.
\end{theorem}

Theorem \ref{th:intro} is an immediate corollary of the following more precise result.

\begin{theorem} \label{th:intro2} 
Let $\Rcal$ be a bordered Riemann surface. Every holomorphic map $f\colon \overline \Rcal\to\c^n$ $(n>1)$ can be approximated, uniformly on compacta in $\Rcal$, by complete proper holomorphic immersions (embeddings if $n>2$) into any open ball in $\c^n$ containing the image $f(\overline \Rcal)$. 
\end{theorem}

By a minor modification of our proof, one can replace balls by arbitrary convex domains. Our method, coupled with the techniques from \cite{DF}, also yields complete proper holomorphic immersions of bordered Riemann surfaces into any Stein manifold of dimension $>1$; see Theorem \ref{th:Stein} in Sec.\ \ref{sec:Stein}.

Our construction is inspired by that of Alarc\'{o}n and L\'{o}pez \cite{AL}, but we use additional complex analytic tools. The examples in \cite{AL} appear after two different deformation procedures, starting from a given holomorphic immersion $f\colon \overline{\Rcal} \to\c^2$; the first one using Mergelyan's theorem, and the second one Runge's theorem. Since the latter provides no information on the placement in $\c^2$ of some parts of the curve, one must shrink the curve to guarantee its boundedness, thereby losing  control of its complex structure. 
In the present paper we replace Runge's theorem by the solution of certain Riemann-Hilbert boundary value problems. This gives sufficient control of the position of the complex curve in the ambient space to avoid shrinking. The use of the Riemann-Hilbert problem in the construction of proper holomorphic maps has a long history; see the 1992 paper of Forstneri\v c and Globevnik \cite{FG} and the introduction and references in \cite{DF}.

Another important tool in our construction is the method of Forstneri\v c and Wold \cite{FW} for exposing boundary points of a complex curve in $\c^n$. This method is used in the inductive step to insure that any curve terminating near a specific finite set of boundary points of $\Rcal$ becomes sufficiently long. On the boundary arcs separated by these points we apply the Riemann-Hilbert method to insure that any curve ending on such an arc is sufficiently long. In a certain precise sense, the first method enables us to localize the length estimates furnished by the second method, thereby preventing the appearance of any shortcuts in the new complex curve. 

These additions allow us to simplify the construction in \cite{AL} and, what is the main point, to avoid changes of the complex structure on the Riemann surface.

It is classical that any open Riemann surface immerses properly holomorphically into $\c^2$ \cite{Bi,Nar} (see also \cite{AL2}), but it is an open question whether it embeds into $\c^2$ \cite{BN}. Forstneri\v c and Wold \cite{FW} gave an affirmative answer for bordered Riemann surfaces whose closures embed holomorphically into $\c^2$, and for all circular domains in $\c$ without punctures \cite{FW2}. (See also \cite{Majcen}.) Since the immersions into $\c^2$ constructed in the present paper may have self-intersections, the following problem remains open.

\begin{question}
Does there exist a complete (properly) embedded complex curve in a ball of $\c^2$?
\end{question}

As we have already mentioned, there exist complete bounded immersed minimal surfaces and null curves with arbitrary topology \cite{AL,FMM}. The first examples of complete properly {\em embedded null curves} with arbitrary topology in a ball of $\c^3$ have been constructed recently in \cite[Corollary 6.2]{AF}. However, the methods in those papers do not enable one to control the complex structure on these surfaces.  In the light of Theorem \ref{th:intro}, the following is therefore a natural question.

\begin{question}
Does every bordered Riemann surface admit a complete conformal minimal immersion into a ball of $\r^3$? Even more, does every such surface admit a complete null holomorphic immersion (or embedding) into a ball of $\c^3$?
\end{question}

By using the theorem of Jones \cite{Jones} one also finds examples of bounded complete complex submanifolds of arbitrary dimension in $\c^N$ for large $N$. In particular, we have the following corollary to Theorem \ref{th:intro}; we wish to thank J.\ E.\ Forn\ae ss for this observation. Let $\mathbb{B}^k$ denote the unit ball of $\c^k$ and $(\mathbb{B}^k)^l \subset \c^{kl}$ the Cartesian product of $l$ copies of $\mathbb{B}^k$.

\begin{corollary}
\label{cor:higher}
Every relatively compact, strongly pseudoconvex domain in a Stein manifold of dimension $n$ admits a complete proper holomorphic immersion into $(\mathbb{B}^2)^{2n} \subset \c^{4n}$, and a complete proper holomorphic embedding into 
$(\mathbb{B}^3)^{2n+1} \subset \c^{6n+3}$. 
\end{corollary}

\begin{proof}
Denote by $\mathbb{D}^k$ the unit polydisc in $\c^k$. According to \cite{DF2010}, any relatively compact strongly pseudoconvex domain $D$ in a Stein manifold of dimension $n$ admits a proper holomorphic immersion $g\colon D\to \mathbb{D}^{2n}$ and a proper holomorphic embedding $h\colon D\hookrightarrow \mathbb{D}^{2n+1}$. (For similar results see also the papers \cite{Dor1995,Low2,Stensones1989}.) Theorem \ref{th:intro} furnishes a proper complete holomorphic immersion $\xi\colon \mathbb{D}\to \mathbb{B}^2$ of the disc into the ball of $\c^2$, and a proper complete holomorphic embedding $\eta \colon \mathbb{D}\to \mathbb{B}^3$ into the ball of $\c^3$. For any integer $k\in\mathbb{N}$ let $\xi^k\colon \mathbb{D}^k \to \c^{2k}$ be defined by 
\[
	\xi^k(z_1,z_2,\ldots,z_k)=(\xi(z_1),\xi(z_2),\ldots, \xi(z_k)),\quad (z_1,\ldots,z_k)\in \mathbb{D}^k.  
\]
Clearly $\xi^k$ immerses $\mathbb{D}^k$ properly into $(\mathbb{B}^2)^k \subset \c^{2k}$. Similarly we define the map $\eta^k\colon \mathbb{D}^k \to \c^{3k}$ which embeds $\mathbb{D}^k$ properly into $(\mathbb{B}^3)^k$. It is now easily verified that the maps $\xi^{2n} \circ g \colon D \to (\mathbb{B}^2)^{2n} \subset \c^{4n}$ and $\eta^{2n+1}\circ h\colon D \to (\mathbb{B}^3)^{2n+1}  \subset \c^{6n+3}$ satisfy the conclusion of the corollary.
\end{proof}

The dimensions in Corollary \ref{cor:higher} are probably far from optimal. It would be interesting to answer the following question.

\begin{question}
Let $n>1$. Does the ball $\mathbb{B}^n\subset \c^n$ admit a proper complete holomorphic immersion (or embedding) into the ball $\mathbb{B}^N$ for some $N>n$? What is the answer if $n=2$ and $N=3,4,\ldots$?
\end{question}

The paper is organized as follows. In Sec.\ \ref{sec:preli} we collect the technical tools. The central part of the paper is Sec.\ \ref{sec:mainlemma} where we state and prove the main technical result, Lemma  \ref{lem:main}. Theorem \ref{th:intro2} is proved in Sec.\ \ref{sec:proof} by a recursive application of Lemma  \ref{lem:main}. In Sec.\ \ref{sec:Stein} we outline the construction of proper complete holomorphic immersions of bordered Riemann surfaces to an arbitrary Stein manifold of dimension $>1$.


\section{Preliminaries}\label{sec:preli}

We denote by $\langle\cdot,\cdot\rangle$, $|\cdot|$, $\dist(\cdot,\cdot)$, and $\length(\cdot)$, respectively, the hermitian inner product, norm, distance, and length on $\c^n$. Hence $\Re \langle\cdotp,\cdotp\rangle$ is the Eulidean inner product on $\r^{2n}\cong\c^n$. 

Given $u\in\c^n$, set $\langle u\rangle^\bot=\{v\in\c^n\colon \langle u,v\rangle=0\}$ and $\span\{u\}=\{\zeta u\colon \zeta\in\c\}$. If $u\neq 0\in\c^n$, then $\langle u\rangle^\bot$ is a complex hyperplane in $\c^n$.

For any compact topological space $K$ and continuous map $f\colon K\to\c^n$, denote by $\|f\|_{0,K}=\sup_{x\in K} |f(x)|$ the maximum norm of $f$ on $K$. If $K$ is a subset of a smooth manifold $M$ and $f$ is of class $\Ccal^r$ for some $r\in\z_+=\n\cup\{0\}$, we denote by $\|f\|_{r,K}$ the $\Ccal^r$-maximum norm of $f$ on $K$, measured with respect the expression of $f$ in a system of local coordinates in some fixed finite open cover of $K$. Similarly we define these norms for maps $M\to X$ to a smooth manifold $X$ by using a fixed cover by coordinate patches on both manifolds. 

Given a topological surface $M$ with boundary, we denote by $bM$ the $1$-dimensional topological manifold determined by the boundary points of $M$. By a {\em domain} in $M$ we mean an open connected subset of $M\setminus  bM$. A surface is said to be {\em open} if it is non-compact and does not contain any boundary points. A {\em Riemann surface} is an oriented surface together with the choice of a complex structure.

\begin{definition}
An open connected Riemann surface, $\Rcal$, is said to be a {\em bordered Riemann surface} if it is the interior of a compact one dimensional complex manifold, $\overline{\Rcal}$, with smooth boundary $b\overline{\Rcal}$ consisting of fini\-tely many closed Jordan curves.
\end{definition}

A domain in a bordered Riemann surface $\Rcal$ is said to be a {\em bordered domain} if it is a bordered Riemann surface, hence with smooth boundary. It is classical that every bordered Riemann surface is biholomorphic to a relatively compact bordered domain in a larger Riemann surface $\widehat \Rcal$ which can be chosen either open or compact.  

Denote by $\Bscr(\Rcal)$ the family of bordered domains $\Mcal\Subset\Rcal$ such that $\Rcal$ is a tubular neighborhood of $\overline{\Mcal}$; that is to say, $\overline{\Mcal}$ is a Runge subset of $\Rcal$ and $\Rcal\setminus  \overline{\Mcal}$ consists of finitely many open annuli. 

Let $\Rcal$ be a bordered Riemann surface and $X$ be a complex manifold. Given a number $r\ge 0$, we denote by $\Ascr^r(\Rcal,X)$ the set of all maps $f\colon \overline \Rcal\to X$ of class $\Ccal^r$ that are holomorphic on $\Rcal$. (This space is defined also for non-integer values of $r$ by H\"older continuity.) For $r=0$ we write $\Ascr^0(\d,X)=\Ascr(\d,X)$. When $X=\c$, we write $\Ascr^r(\Rcal,\c) = \Ascr^r(\Rcal)$ and $\Ascr^0(\Rcal) = \Ascr(\Rcal)$. Note that $\Ascr^r(\Rcal,\c^n) = \Ascr^r(\Rcal)^n$ is a complex Banach space. For any complex manifold $X$ and any number $r\ge 0$, the space $\Ascr^r(\Rcal,X)$ carries a natural structure of a complex Banach manifold \cite{F2007}. 
 
Let $\Iscr(\Rcal,\c^n)$ denote the set of all $\Ccal^1$ immersions $\overline{\Rcal}\to\c^n$ that are holomorphic in the interior. We shall also write $\Iscr(\Rcal)$ when the target is clear from the context. 

Given an immersion $f\colon \Rcal\to\c^n$, we denote by $\sigma_f^2$ the Riemannian metric in $\Rcal$ induced by the Euclidean metric of $\c^n$ via $f$; that is, 
\[
		\sigma_f^2(p,v) = \langle df_p(v), df_p(v)\rangle, \quad p\in \Rcal,\ v\in T_p\Rcal.
\]
If $f$ is holomorphic, then $\sigma_f^2$ is a conformal metric on $\Rcal$. We denote by $\dist_{(\Rcal,f)}(\cdot,\cdot)$ the distance function in the Riemannian surface $(\Rcal,\sigma_f^2)$. 

A curve $\gamma\colon [0,1)\to \Rcal$ is said to be {\em divergent} if the map $\gamma$ is proper; that is, if the point $\gamma(t)$ leaves any compact subset of $\Rcal$ when $t\to 1$.

\begin{definition}
\label{def:complete_imm}
An immersion $f\colon \Rcal\to\c^n$ is said to be {\em complete} if $(\Rcal,\sigma_f^2)$ is complete as a Riemannian surface; that is to say, if the image curve $f\circ \gamma$ in $\c^n$ has infinite length for any connected divergent curve $\gamma$ in $\Rcal$. 
\end{definition}


\subsection{A Riemann-Hilbert problem}\label{sec:RH}

We shall need approximate solutions of certain Riemann-Hilbert problems over the unit disc $\d=\{\zeta\in\c\colon |\zeta|<1\}$.  Such results have been used by several authors; the precise version of the following lemma is taken from the papers \cite{DF,DF2}.

\begin{lemma} 
\label{lem:Hilbert}
Fix an integer $n\in\n$. Let $f\in\Ascr(\d)^n$, and let $g\colon b\overline{\d}\times\overline{\d}\to\c^n$ be a continuous map such that for every fixed $\zeta\in b\overline{\d}$ we have $g(\zeta,\cdot)\in \Ascr(\d)^n$ and $g(\zeta,0)=f(\zeta)$. Given numbers $\epsilon>0$ and $0<r<1$, there exist a number $r'\in [r,1)$ and a map $h\in\Ascr(\d)^n$ satisfying the following conditions:
\begin{itemize}
\item $\dist(h(\zeta),g(\zeta,b\overline{\d}))<\epsilon$ for all $\zeta\in b\overline{\d}$,
\item $\dist(h(\rho\zeta),g(\zeta,\overline{\d}))<\epsilon$ for all $\zeta\in b\overline{\d}$ and $\rho\in [r',1)$, and
\item $|h(\zeta)-f(\zeta)|<\epsilon$ when $|\zeta|\leq r'$.
\end{itemize}

In particular, if $J$ is a compact arc in the circle $ b\overline{\d}$ and $g(\zeta,\cdot)=f(\zeta)$ is the constant disc for all points $\zeta\in b\overline{\d}\setminus  J$, then for any open neighborhood $C$ of $J$ in $\overline{\d}$ one can choose $h$ such that $\|h-f\|_{0,\overline{\d}\setminus  C}<\epsilon$.
\end{lemma}

A characteristic feature of the map $h$ in the above lemma is that its boundary values must spin very fast in a small neighborhood of the torus in $\c^n$ that is formed by the boundaries of the analytic discs $g(\zeta,\cdotp)$ for $|\zeta|=1$. In the model situation in $\c^2$, with $f(\zeta)=(\zeta,0)$ and $g(\zeta,z)=(\zeta,z)$ for $|\zeta|=1$ and $|z|\le 1$, a suitable (exact!) solution is the map $h(\zeta)=(\zeta,\zeta^N)$ for large values of $N\in \n$. Indeed, the proof of the lemma amounts to a reduction to this model situation.


\subsection{A gluing lemma}\label{sec:gluing}

An important technical step in our construction is to glue pairs of holo\-mor\-phic mappings on special geometric configurations in a Riemann surface. We first explain this gluing for maps to $\c^n$; in the next subsection we address the correspoding problem for maps to an arbitrary complex manifold. (The general case is only used in Sec.\ \ref{sec:Stein}.)

\begin{definition}
\label{def:Cartan}
A {\em Cartan decomposition} of a bordered Riemann surface $\overline \Rcal$ is a pair $(A,B)$ of compact domains with smooth boundaries in $\overline\Rcal$ such that $C:=A\cap B$ also has smooth boundary, and we have
\[
	A\cup B =\overline \Rcal,\qquad \overline{A\setminus  B}\cap \overline{B\setminus  A} =\emptyset. 
\]
\end{definition}

Compare with the notion of a {\em Cartan pair} (Chapter 5 in \cite{F2011}). In applications in this paper, one of the two sets $A,B$ will be a disjoint union of finitely many discs. 

The following well known lemma is a solution of the Cousin-I problem with bounds; c.f.\ Lemma 5.8.2 in \cite{F2011}.
The operators $\Acal$ and $\Bcal$ are obtained by using a smooth cut-off function for the pair $(A,B)$ and a bounded solution operator for the $\overline\partial$-problem on $\overline \Rcal$.

%
%
%
\begin{lemma}
\label{linear-splitting}
Let $A\cup B =\overline \Rcal$ be a Cartan decomposition of a bordered Riemann surface $\overline \Rcal$. Set $C:=A\cap B$. For every integer $r\in\z_+$ there exist bounded linear operators $\Acal \colon \Ascr^r(C) \to \Ascr^r(A)$, $\Bcal \colon \Ascr^r(C) \to \Ascr^r(B)$, satisfying the condition
\[ 
        c = \Acal c - \Bcal c \qquad \forall c\in \Ascr^r(C).
\]
\end{lemma}

\begin{corollary}
\label{linear-gluing}
{\rm (Assumptions as in Lemma \ref{linear-splitting}.)} For every pair of integers $n\in \n$ and $r\in\z_+$ there is a constant $M=M_{n,r}>0$ with the following property. Given a pair of mappings $f\in \Ascr^r(A)^n$, $g\in \Ascr^r(B)^n$, there exists $F\in \Ascr^r(\overline \Rcal)^n$ such that
\[
	||F-f||_{r,A} \le M ||f-g||_{r,C},\qquad ||F-g||_{r,B} \le M ||f-g||_{r,C}.
\]
\end{corollary}

The corollary follows immediately from Lemma \ref{linear-splitting} by setting 
\[
	c=f|_C-g|_C \in \Ascr^r(C)^n, \quad a=\Acal c\in \Ascr^r(A)^n,\quad b=\Bcal c\in \Ascr^r(B)^n,
\]
and noting that as a consequence we have $f-a=g-b$ on $C$. Hence the maps $f-a\colon A\to\c^n$ and $g-b\colon B\to\c^n$ amalgamate into a map $F\colon \overline \Rcal\to\c^n$ with the stated properties. The constant $M$ depends on the  norms of the operators $\Acal,\Bcal$ in Lemma \ref{linear-splitting}.

\subsection{Gluing holomorphic sprays}\label{sec:sprays}

The simple gluing method in the previous subsection no longer works in the absence of a linear structure on the target manifold. It can be replaced by the method of gluing pairs of holomorphic sprays of maps that was developed by Drinovec Drnov\v sek and Forstneri\v c \cite{DF,DF2,F2007}; this method applies to maps with values in an arbitrary complex manifold. We recall the main results and refer for further details to the cited papers, or to Sect.\ 5.9 in \cite{F2011}.

\begin{definition}\label{def:spray}
Let $r\in\z_+$. A {\em spray of maps} $\overline \Rcal\to X$ of class $\Ascr^r(\Rcal,X)$ is a $\Ccal^r$ map $F\colon P \times \overline{\Rcal}\to X$ that is holomorphic on $P\times\Rcal$, where $P$ is a domain of a complex Euclidean space $\c^m$ containing the origin $0\in\c^m$. The spray is said to be {\em dominating} if the partial differential $\frac{\partial}{\partial t} F(t,z) \colon T_t\c^m\cong \c^m \to T_{F(t,z)}X$ is surjective for all $(t,z)\in P\times \overline \Rcal$. The map $f=F(0,\cdot)\colon \Rcal\to X$ is the {\em core map}, and $P$ is the {\em parameter set} of the spray $F$. 
\end{definition}

A spray of maps in the above definition can be viewed as a holomorphic map from the domain $P\subset \c^m$ to the Banach manifold $\Ascr^r(\Rcal,X)$.

\begin{lemma}{\rm (Existence of sprays, \cite{DF}.)} \label{lem:Espray}
Let $r\in\z_+$. For every map $f\in \Ascr^r(\Rcal,X)$ there exists a dominating spray of maps $F\colon P\times\overline{\Rcal}\to X$ of class $\Ccal^r$ such that $F(0,\cdot)=f$.  
\end{lemma}

When $X=\c^n$, we can simply take $F(t,z)=f(z)+t$ for $t\in\c^n$ to get a dominating spray of maps $\c^n\times \Rcal  \to \c^n$ with the core map $f$.

\begin{lemma} {\rm (Gluing sprays of maps, \cite{DF}.)} \label{lem:Gspray}
Let $r\in\z_+$. Let $(A,B)$ be a Cartan decomposition of a bordered Riemann surface $\overline \Rcal$ (Def.\ \ref{def:Cartan}), and set $C=A\cap B$. Given a domain $0\in P_0\subset\c^m$ and a dominating spray of maps $F_A \colon P_0\times A \to X$ of class $\Ccal^r$, there exists a domain $P\Subset P_0\subset\c^m$ containing $0\in\c^m$ such that the following holds.  

For every spray of maps $F_B\colon P_0\times B \to X$ of class $\Ccal^r$ that is sufficiently $\Ccal^r$-close to $F_A$ on $P_0\times C$ there exists a spray of maps $F\colon P\times \overline \Rcal \to X$ of class $\Ccal^r$ such that:
\begin{itemize}
\item the restriction $F \colon P\times A\to X$ is close to $F_A$ in the $\Ccal^r$-topology (depending on the $\Ccal^r$-distance between $F_A$ and $F_B$ on $P_0\times C$), and
\item $F(t,z)\in \{F_1(s,z)\colon s\in P_0\}$ for all $t\in P$ and $z\in B$.
\end{itemize}    
\end{lemma}

\begin{remark}
\label{rem:r=0}
For applications in this paper it suffices to use these gluing methods in the basic case $r=0$, i.e., with continuity up to the boundary. This is because every map $f\in \Ascr(\Rcal,X)$ can be approximated, uniformly on $\overline \Rcal$, by maps holomorphic in a neighborhood of $\overline \Rcal$ in a larger Riemann surface \cite{DF}.
\end{remark}


\section{The Main Lemma}\label{sec:mainlemma}

In this section we prove an approximation result, Lemma \ref{lem:main}, which is the core of this paper; Theorem \ref{th:intro} will follow by a standard recursive application of it. We focus on maps to $\c^2$, although the same method works for any $\c^n$, $n>1$.

Let $\b(s)=\{u\in\c^2\colon |u|<s\}$ denote the open ball of radius $s>0$ in $\c^2$, centered at the origin, and let $\overline{\b}(s)=\{u\in\c^2\colon |u|\leq s\}$ be the corresponding closed ball. Recall from Sec.\ \ref{sec:preli} that $\Bscr(\Rcal)$ denotes the set of all bordered Runge domains $\Mcal\Subset\Rcal$, and $\Iscr(\Rcal)$ is the set of all immersions $\overline \Rcal\to \c^2$ that are holomorphic in $\Rcal$.

%
%
%
%
\begin{lemma}\label{lem:main}
Let $\Rcal$ be a bordered Riemann surface, let $\Mcal\in\Bscr(\Rcal)$, let $z_0$ be a point in $\Mcal$, let $f\in \Iscr(\Rcal)$, and let $\epsilon$, $\rho$, and $s>\epsilon$ be positive constants. Assume that
\begin{enumerate}[\sf (i)]
\item $f(\overline{\Rcal} \setminus  \Mcal)\subset \b(s) \setminus  \overline{\b}(s-\epsilon)$, and
\item $\dist_{(\Rcal,f)}(z_0,b\overline{\Rcal})>\rho$.
\end{enumerate}
Then, for any $\hat{\epsilon}>0$ and $\delta>0$, there exists an immersion $\hat{f}\in \Iscr(\Rcal)$ 
satisfying the following conditions: 
\begin{enumerate}[\sf (L1)]
\item $\|\hat{f}-f\|_{1,\overline{\Mcal}}<\hat{\epsilon}$,
\item $\hat{f}(b\overline{\Rcal})\subset \b(\hat{s}) \setminus  \overline{\b}(\hat{s}-\hat{\epsilon})$, where $\hat{s} =\sqrt{s^2+\delta^2}$,
\item $\hat{f}(\overline{\Rcal}\setminus  \Mcal)\cap\overline{\b}(s-\epsilon)=\emptyset$, and
\item $\dist_{(\Rcal,\hat{f})}(z_0,b\overline{\Rcal})>\hat{\rho} :=\rho+\delta$.
\end{enumerate}
\end{lemma}

The lemma asserts that any holomorphic immersion $f\colon \overline{\Rcal}\to\c^2$ such that $f(b\overline{\Rcal})$ lies near a given sphere in $\c^2$ (see {\sf (i)}) can be deformed into another holomorphic immersion $\hat{f}\colon \overline{\Rcal}\to\c^2$ mapping the boundary $b\overline{\Rcal}$ as close as desired to a larger sphere in $\c^2$ (see {\sf (L2)}). The deformation is strong near the boundary of $\Rcal$ and is arbitrarily small on a given compact subset of $\Rcal$. In this process both the extrinsic and the intrinsic diameters of the surface grow, but their respective growths are related in a Pythagoras' way; compare the bounds $\hat{s}$ in {\sf (L2)} and $\hat{\rho}$ in {\sf (L4)}. This link, which is the key to obtain completeness while preserving boundedness, follows the spirit of Nadirashvili's original construction of complete bounded minimal surfaces in $\r^3$  \cite{Nad}. The novelty of Lemma \ref{lem:main} is that do not have to modify the complex struct
 ure on $\Rcal$ in order to keep the extrinsic diameter of the curve suitably bounded. This represents a clear difference with respect to previous results, and is the main improvement obtained in this work.

The proof of Lemma \ref{lem:main} consists of three main steps which we now describe.

In the first step (Subsec.\ \ref{sec:step1}) we split the boundary $b\overline{\Rcal}$ into finitely many compact Jordan arcs $\alpha_{i,j}$, with endpoints $p_{i,j-1}$ and $p_{i,j}$, so that deformations of $f$ near $\alpha_{i,j}$, preserving the complex direction $f(p_{i,j})\in\c^2$, keep the image of $\alpha_{i,j}$ disjoint from the ball $\overline{\b}(s-\epsilon)$; at the same time, the length of any segment in $\c^2$, orthogonal to $f(p_{i,j})$ and connecting $f(\alpha_{i,j})$ to the boundary of $\b(\hat{s})$, is bigger than $\delta$. 

In the second step (Subsec.\ \ref{sec:step2}) we deform $f$ into another holomorphic immersion $f_0\colon \overline{\Rcal}\to\c^2$. The deformation is large near the points $p_{i,j}$ and is small elsewhere.
The new immersion $f_0$ maps a neighborhood of each point $p_{i,j}$ close to the boundary of the ball $\b(\hat{s})$ (see {\sf (d6)}), and the length of the image of any curve in $\Rcal$ connecting $z_0$ to $p_{i,j}$ is larger than $\hat{\rho}$ (see {\sf (d9)}). Roughly speaking, $f_0$  satisfies items {\sf (L1)}, {\sf (L3)}, and {\sf (L4)} in Lemma \ref{lem:main}, but it meets {\sf (L2)} only near the points $p_{i,j}$. This step is inspired by the method of {\em exposing boundary points} of a complex curve in $\c^2$ that was developed by Forstneri\v c and Wold \cite{FW} (see also Sec.\ 8.9 in \cite{F2011}).   

In the third step (Subsecs.\ \ref{sec:step3} and \ref{sec:step4}) we work on the part $\beta_{i,j}^2$ of the arc $\alpha_{i,j}$ where $f_0$ does not meet the condition {\sf (L2)}; that is, outside a neighborhood of the endpoints $p_{i,j-1}$ and $p_{i,j}$. We use Lemma \ref{lem:Hilbert} to solve a suitable Riemann-Hilbert problem over a closed disc $\overline{D}_{i,j}$ containing the arc $\beta_{i,j}^2$ in its boundary. This gives a holomorphic map $h_{i,j}\colon \overline{D}_{i,j}\to\c^2$ which satisfies {\sf (L2)} over $\beta_{i,j}^2$ (see {\sf (f1)}), is close to $f_0$ outside a small neighborhood of $\beta_{i,j}^2$ in $\overline{D}_{i,j}$ (see {\sf (f3)}), and whose orthogonal projection onto the complex line spanned by $f_0(p_{i,j})$ is close to the projection  of $f_0$ (see {\sf (f2)}). We then glue $f_0$ and the $h_{i,j}$'s into a new immersion $\hat{f}\in \Iscr(\Rcal)$.

Finally, in Subsec.\ \ref{sec:step5} we verify that the map $\hat{f}$ obtained in this way satisfies Lemma \ref{lem:main}.


\smallskip \noindent \bf Proof of Lemma \ref{lem:main}. \rm
Replacing $\Mcal$ by a larger bordered domain in $\Bscr(\Rcal)$ if necessary, we may assume that 
\begin{equation}\label{eq:lemma}
\dist_{(\Rcal,f)}(z_0,b\overline{\Mcal})>\rho;
\end{equation}
see the strict inequality {\sf (ii)} and recall that $b\overline{\Rcal}$ is compact. Moreover, we can realize $\Rcal$ as a bordered domain in an open Riemann surface $\widehat{\Rcal}$ such that $\Rcal \in \Bscr(\widehat{\Rcal})$. Then, by Mergelyan's theorem, we may assume that $f\in\Iscr(\Rcal)$ extends to a holomorphic immersion $f\in\Iscr(\widehat{\Rcal})$.


\subsection{Splitting $b\overline{\Rcal}$}\label{sec:step1}

Note that for any $u\in\c^2\setminus \{0\}$, $u+\langle u\rangle^\bot$ is the affine complex line passing through $u$ and orthogonal to $u$. From {\sf (i)}, the definition of $\hat{s}$ in {\sf (L2)}, and Pythagoras' theorem, one has
\[
\dist\big( u\,,\, b{\b}(\hat{s}) \cap (u+\langle u\rangle^\bot)\big)>\delta \quad \forall u\in f(b\overline{\Rcal}).
\]
Then, by continuity and up to decreasing $\hat{\epsilon}>0$ if necessary, any point $u\in f(b\overline{\Rcal})$ admits an open neighborhood $U_u$ in $\c^2$ such that
\begin{equation}\label{eq:Pytha}
\dist\big( v\,,\, (\b(\hat{s}) \setminus  \overline{\b}(\hat{s}-\hat{\epsilon}))\cap (w+\langle z\rangle^\bot)\big)>\delta\quad \forall v,w,z\in U_u.
\end{equation}
Condition {\sf (i)} also implies that
\[
(u+\langle u\rangle^\bot)\cap\overline{\b}(s-\epsilon)=\emptyset\quad \forall u\in f(b\overline{\Rcal}).
\]
Hence any point $u\in f(b\overline{\Rcal})$ admits an open neighborhood $V_u\subset \c^2$ such that
\begin{equation}\label{eq:far}
(v+\langle w\rangle^\bot)\cap\overline{\b}(s-\epsilon)=\emptyset\quad \forall v,w\in V_u;
\end{equation}
we are taking into account the compactness of $\overline{\b}(s-\epsilon)$. Set
\[
\Uscr_u=U_u\cap V_u,\; u\in f(b\overline{\Rcal}), \qquad  \Uscr=\{\Uscr_u\colon u\in f(b\overline{\Rcal})\}.
\]

Denote by $\alpha_1,\ldots,\alpha_\igot$ the connected boundary curves of $\overline{\Rcal}$, so $b\overline{\Rcal}=\cup_{i=1}^\igot \alpha_i$. Since $\Rcal\in\Bscr(\widehat{\Rcal})$, the $\alpha_i$'s are pairwise disjoint smooth closed Jordan curves in $\widehat{\Rcal}$. As $\Uscr$ is an open covering of the compact set $f(b\overline{\Rcal})\subset\c^2$, there exist an integer $\jgot \geq 3$ and compact connected subarcs $\{\alpha_{i,j}\colon (i,j) \in \{1,\ldots,\igot\}\times\z_\jgot\}$, where $\z_\jgot=\{0,\ldots,\jgot-1\}$ denotes the additive cyclic group of integers modulus $\jgot$, satisfying the following conditions:
\begin{enumerate}[\sf ({a}1)]
\item $\cup_{j=1}^\jgot\alpha_{i,j}=\alpha_i$,
\item $\alpha_{i,j}$ and $\alpha_{i,j+1}$ have the common endpoint $p_{i,j}$ and are otherwise disjoint,
\item there exist points $a_{i,j}\in \b(\hat{s})\setminus \overline{\b}(\hat{s}-\hat{\epsilon})$ such that 
\[
(a_{i,j}+\langle f(p_{i,k})\rangle^\bot)\cap \overline{\b}(\hat{s}-\hat{\epsilon})=\emptyset \quad \forall k\in\{j,j+1\}, \quad \text{and}
\] 
\item for every $(i,j) \in \{1,\ldots,\igot\}\times\z_\jgot$ there exists a set $\Uscr_{i,j}\in\Uscr$ containing the curve $f(\alpha_{i,j})$. In particular $f(p_{i,j})\in \Uscr_{i,j}\cap\Uscr_{i,j+1}$.
\end{enumerate}

A splitting with these properties is found by choosing the arcs $\alpha_{i,j}$ such that their images $f(\alpha_{i,j}) \subset \c^2$ have small enough diameter.  


\subsection{Stretching from the points $p_{i,j}$}\label{sec:step2}

We adapt to our needs the method from \cite{FW} that was used in that paper for exposing boundary points.  (See also Sections 8.8 and 8.9 in \cite{F2011}.) Our goal here is a different one: we wish to modify the immersion so that the images of a certain finite collection of arcs in $\Rcal$, terminating at points of $b\overline \Rcal$, become very long in $\c^2$.

Recall that $\overline \Rcal$ is a compact bordered domain in $\widehat \Rcal$. For every $(i,j)\in\{1,\ldots,\igot\}\times\z_\jgot$ we choose an embedded real analytic arc $\gamma_{i,j}\subset\widehat{\Rcal}$ that is attached to $\overline{\Rcal}$ at the endpoint $p_{i,j}$ and intersects $b\overline \Rcal$ transversely there, and is otherwise disjoint from $\overline{\Rcal}$. We insure that the arcs $\gamma_{i,j}$ are pairwise disjoint. Let $q_{i,j}$ denote the other endpoint of $\gamma_{i,j}$. We split $\gamma_{i,j}$ into compact subarcs $\gamma_{i,j}^1$ and $\gamma_{i,j}^2$, with a common endpoint, so that $p_{i,j}\in\gamma_{i,j}^1$ and $q_{i,j}\in\gamma_{i,j}^2$.

Next we choose compact smooth embedded arcs $\lambda_{i,j} \subset \c^2$ satisfying the following properties:
\begin{enumerate}[\sf ({b}1)]
\item $\lambda_{i,j}\subset \b(\hat{s})\setminus \overline{\b}(s-\epsilon)$,
\item $\lambda_{i,j}$ is split into compact subarcs $\lambda_{i,j}^1$ and $\lambda_{i,j}^2$ with a common endpoint,
\item $\lambda_{i,j}^1$ agrees with the arc $f(\gamma_{i,j})$ near the endpoint $f(p_{i,j})$; recall that $f\in\Iscr(\widehat{\Rcal})$, 
\item $\lambda_{i,j}^1\subset \Uscr_{i,j}\cap\Uscr_{i,j+1}$; see properties {\sf (b3)} and {\sf (a4)},
\item if $J\subset\lambda_{i,j}^1$ is a Borel measurable subset, then
\begin{eqnarray*}
\min\{\length(\pi_{i,j}(J)),\length(\pi_{i,j+1}(J))\}&+&\\
\min\{\length(\pi_{i,j}(\lambda_{i,j}^1\setminus J)),\length(\pi_{i,j+1}(\lambda_{i,j}^1\setminus J))\}&>&\delta,
\end{eqnarray*}
where $\pi_{i,k}\colon \c^2\to\span\{f(p_{i,k})\}\subset\c^2$ denotes the orthogonal projection (observe that {\sf (i)} implies $f(p_{i,k})\neq 0$),
\item $(\lambda_{i,j}^2+\langle f(p_{i,k})\rangle^\bot)\cap \overline{\b}(s-\epsilon)=\emptyset$ for $k\in \{j,j+1\}$; see {\sf (a4)}, {\sf (b4)}, and (\ref{eq:far}), and
\item the endpoint $v_{i,j}$ of $\lambda_{i,j}$ contained in the subarc $\lambda_{i,j}^2$ lies in the  shell $\b(\hat{s})\setminus \overline{\b}(\hat{s}-\hat{\epsilon})$ and satisfies $(v_{i,j}+\langle f(p_{i,k})\rangle^\bot)\cap \overline{\b}(\hat{s}-\hat{\epsilon})=\emptyset$ for $k\in\{j,j+1\}$; see {\sf (a3)}.
\end{enumerate}

The existence of such arcs $\lambda_{i,j}$ is an easy exercise. Property {\sf (b1)} is compatible with the rest thanks to 
conditions {\sf (a3)} and {\sf (a4)}. To enjoy {\sf (b4)} and {\sf (b5)}, the curve $\lambda_{i,j}^1$ must be highly oscillating and with small diameter in $\c^2$. On the other hand, to satisfy {\sf (b6)} and {\sf (b7)}, one can simply take $\lambda_{i,j}^2$ to be a straight line segment in $\b(\hat{s})\setminus \overline{\b}(s-\epsilon)$, connecting $\Uscr_{i,j}\cap\Uscr_{i,j+1}$ (see {\sf (b4)}) to a point $v_{i,j}$ as those given by {\sf (a3)}.

By property {\sf (b3)} we can find a smooth map $f_{\rm e}\colon \widehat \Rcal \to\c^2$ which agrees with $f$ in an open neighborhood of $\overline \Rcal$, and which maps the arcs $\gamma_{i,j}^1$ and $\gamma_{i,j}^2$ diffeomorphically onto the corresponding arcs $\lambda_{i,j}^1$ and $\lambda_{i,j}^2$ for all $(i,j)$.

The compact set $K:=\overline{\Rcal}\cup(\cup_{i,j}\gamma_{i,j})\subset \widehat{\Rcal}$ clearly admits a basis of open neighborhoods that are Runge in $\widehat \Rcal$. By Mergelyan's theorem one can therefore approximate $f_{\rm e}$, uniformly on an open neighbor\-hood of $\overline{\Rcal}$ and in the $\Ccal^1$-topology on each of the arcs $\gamma_{i,j}$, by a holomorphic immersion $\tilde{f}\colon \widehat \Rcal \to\c^2$.

We shall now apply \cite[Theorem 2.3]{FW} (see also \cite[Theorem 8.8.1]{F2011}). Choose a small open neighborhood $V\subset \widehat \Rcal$ of $K$. Further, for every $(i,j)$ we choose a pair of small neighborhoods 
$W'_{i,j} \Subset W_{i,j} \Subset \widehat \Rcal\setminus \overline \Mcal$ of the point $p_{i,j}$ and a neighborhood $V_{i,j} \Subset \widehat \Rcal \setminus \overline \Mcal$ of the arc $\gamma_{i,j}$. The cited theorem furnishes a smooth diffeomorphism $\phi\colon \overline {\Rcal}\to\phi (\overline{\Rcal})\subset V$ satisfying the following properties (see Fig.\ \ref{fig:Phi}):
\begin{itemize}
\item $\phi \colon \Rcal \to  \phi(\Rcal)$ is biholomorphic,
\item $\phi$ is as close as desired to the identity in the $\Ccal^1$-topology on $\overline \Rcal \setminus \cup_{i,j} W'_{i,j}$, and 
\item $\phi(p_{i,j}) = q_{i,j}$ and $\phi(\overline \Rcal \cap W'_{i,j}) \subset W_{i,j} \cup V_{i,j}$ for all $(i,j)\in\{1,\ldots,\igot\}\times\z_\jgot$. 
\end{itemize}
Intuitively speaking, $\phi$ hardly changes $\overline \Rcal$ outside small neighborhoods of the chosen boundary points $p_{i,j}$, while at each of these points it creates a narrow spike reaching up to the opposite endpoint $q_{i,j}$ of the arc $\gamma_{i,j}$. (See Fig.\ 8.1 in \cite[p.\ 367]{F2011}.) Although it is not claimed in the cited sources that the arc $\gamma_{i,j}$ actually belongs to the image $\phi(\overline \Rcal)$, an inspection of the construction shows that we can move it very slightly to a nearby arc  $\gamma'_{i,j}\subset \widehat\Rcal$ (keeping its endpoints fixed) to obtain this property. If the new arc $\gamma'_{i,j}$ is close enough to $\gamma_{i,j}$ (which can be achieved by a suitable choice of $\phi$), we can also replace the arc $\lambda_{i,j} \subset \c^2$ by the image $\tilde f (\gamma'_{i,j})$ without disturbing any of the properties {\sf (b1)}--{\sf (b7)}. In the sequel we drop the primes and assume that, in addition to the above, we have
\begin{itemize}
\item $\gamma_{i,j}\setminus\{q_{i,j}\}\subset\phi(\Rcal \setminus \overline{\Mcal})$ for all $(i,j)$.
\end{itemize}

\begin{figure}[ht]
    \begin{center}
    \resizebox{0.9\textwidth}{!}{\includegraphics{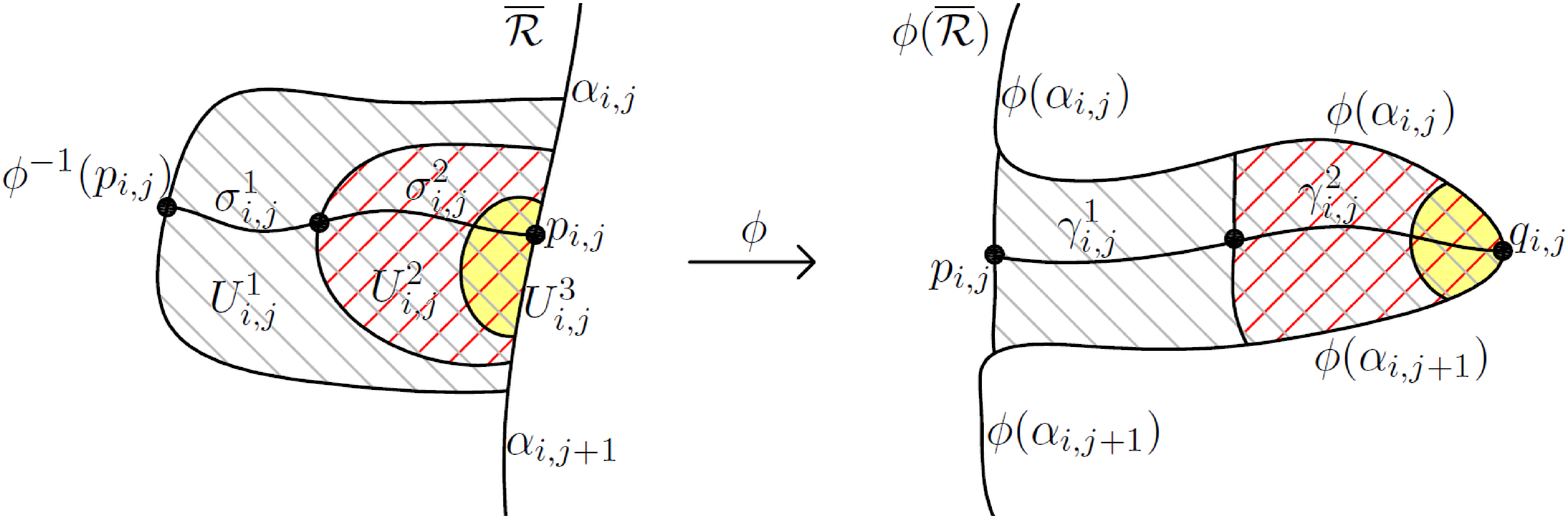}}
        \end{center}
        \vspace{-0.25cm}
\caption{The diffeomorphism $\phi$.}\label{fig:Phi}
\end{figure} 

Set 
\begin{equation} \label{eq:sigma}
\sigma_{i,j}^k=\phi^{-1}(\gamma_{i,j}^k)\; \text{for $k\in\{1,2\}$}, \qquad \sigma_{i,j}=\phi^{-1}(\gamma_{i,j})=\sigma_{i,j}^1\cup \sigma_{i,j}^2.
\end{equation}
(See Fig.\ \ref{fig:Phi}.) If the approximations desribed above are close enough, then the composition
\begin{equation}
\label{eq:f0}
		f_0=\tilde{f}\circ\phi\colon \overline{\Rcal}\to\c^2
\end{equation}
is a holomorphic immersion satisfying the following properties:
\begin{enumerate}[\sf ({c}1)]
\item $f_0(p_{i,j})\in \b(\hat{s})\setminus \overline{\b}(\hat{s}-\hat{\epsilon})$ and $(f_0(p_{i,j})+\langle f(p_{i,k})\rangle^\bot)\cap \overline{\b}(\hat{s}-\hat{\epsilon})=\emptyset$ for $k\in\{j,j+1\}$ and $(i,j)\in\{1,\ldots,\igot\}\times\z_\jgot$; recall that $f_0(p_{i,j})=\tilde{f}(q_{i,j})\approx f_{\rm e}(q_{i,j})=v_{i,j}$ and see {\sf (b7)},
\item $\|f_0-f\|_{1,\overline{\Mcal}}<\hat{\epsilon}/2$,
\item $f_0(\overline{\Rcal}\setminus \Mcal)\subset \b(\hat{s})\setminus \overline{\b}(s-\epsilon)$; see the open conditions {\sf (i)} and {\sf (b1)} and recall that $\lambda_{i,j}$ is compact, and
\item $\dist_{(\Rcal,f_0)}(z_0,b\overline{\Mcal})>\rho$; see the strict inequality (\ref{eq:lemma}).
\end{enumerate}

Furthermore, if the above approximations are close enough and the neighborhood $V$ of $K$ (containing the image $\phi(\overline \Rcal)$) is small enough, then taking into account properties {\sf (b1)}--{\sf (b7)} and {\sf (c1)}--{\sf (c4)}, one can easily find simply connected neigh\-bor\-hoods $U_{i,j}^3\Subset U_{i,j}^2\Subset U_{i,j}^1$ of the point $p_{i,j}$ in $\overline{\Rcal}\setminus \overline{\Mcal}$ for any $(i,j)\in\{1,\ldots,\igot\}\times\z_\jgot$ satisfying the following properties  (see Figs.\ \ref{fig:Phi} and \ref{fig:R}):
\begin{enumerate}[\sf ({d}1)]
\item $\overline{U}_{i,j}^1\cap \overline{U}_{i,k}^1=\emptyset$ if $k\neq j$,
\item $\overline{U}_{i,j}^1\cap \alpha_{i,k}=\emptyset$ if $k\notin\{j,j+1\}$,
\item $\overline{U}_{i,j}^1\cap \alpha_{i,k}$ is a connected compact Jordan arc for  $k\in\{j,j+1\}$,
\item $\beta_{i,j}^k :=\overline{\alpha_{i,j}\setminus (U_{i,j}^{k+1}\cup U_{i,j-1}^{k+1})}$ are connected compact Jordan arcs for $k\in\{1,2\}$, and $f_0(\beta_{i,j}^1)\subset\Uscr_{i,j}$; see {\sf (a4)},
\item $\sigma_{i,j}^1\subset \overline{U_{i,j}^1\setminus U_{i,j}^2}$, $\sigma_{i,j}^2\subset\overline{U}_{i,j}^2$ (see (\ref{eq:sigma})), and $\phi^{-1}(p_{i,j})\in (b\overline{U}_{i,j}^1)\cap\Rcal$,
\item $f_0(\overline{U}_{i,j}^3)\subset\b(\hat{s})\setminus \overline{\b}(\hat{s}-\hat{\epsilon})$ and $(f_0(\overline{U}_{i,j}^3)+\langle f(p_{i,k})\rangle^\bot)\cap \overline{\b}(\hat{s}-\hat{\epsilon})=\emptyset$ for  $k\in\{j,j+1\}$; see {\sf (c1)},
\item $\left( f_0(\overline{U}_{i,j}^2) + \langle f(p_{i,k})\rangle^\bot \right) \cap \overline{\b}(s-\epsilon)=\emptyset$ for all $k\in \{j,j+1\}$; see {\sf (b6)},
\item $f_0(\overline{U_{i,j}^1\setminus U_{i,j}^2})\subset \Uscr_{i,j}\cap\Uscr_{i,j+1}$; see {\sf (b4)}, and
\item if $\gamma\subset \overline{U}_{i,j}^1$ is an arc connecting $\Rcal\setminus U_{i,j}^1$ and $\overline{U}_{i,j}^2$, and $J\subset \gamma$ is a Borel measurable subset, then (see {\sf (b5)}) we have 
\begin{eqnarray*}
\min\{\length(\pi_{i,j}(f_0(J))),\length(\pi_{i,j+1}(f_0(J)))\}&+&\\
\min\{\length(\pi_{i,j}(f_0(\gamma \setminus J))),\length(\pi_{i,j+1}(f_0(\gamma\setminus J)))\}&>&\delta.
\end{eqnarray*}
\end{enumerate}
In fact, the above properties hold if $U_{i,j}^3$ is chosen sufficiently small around the point $p_{i,j}$, $U_{i,j}^2$ is sufficiently small around the arc $\sigma_{i,j}^2$, and $U_{i,j}^1$ is sufficiently small around the arc $\sigma_{i,j}$ (see (\ref{eq:sigma})). 

\begin{figure}[ht]
    \begin{center}
    \resizebox{0.70 \textwidth}{!}{\includegraphics{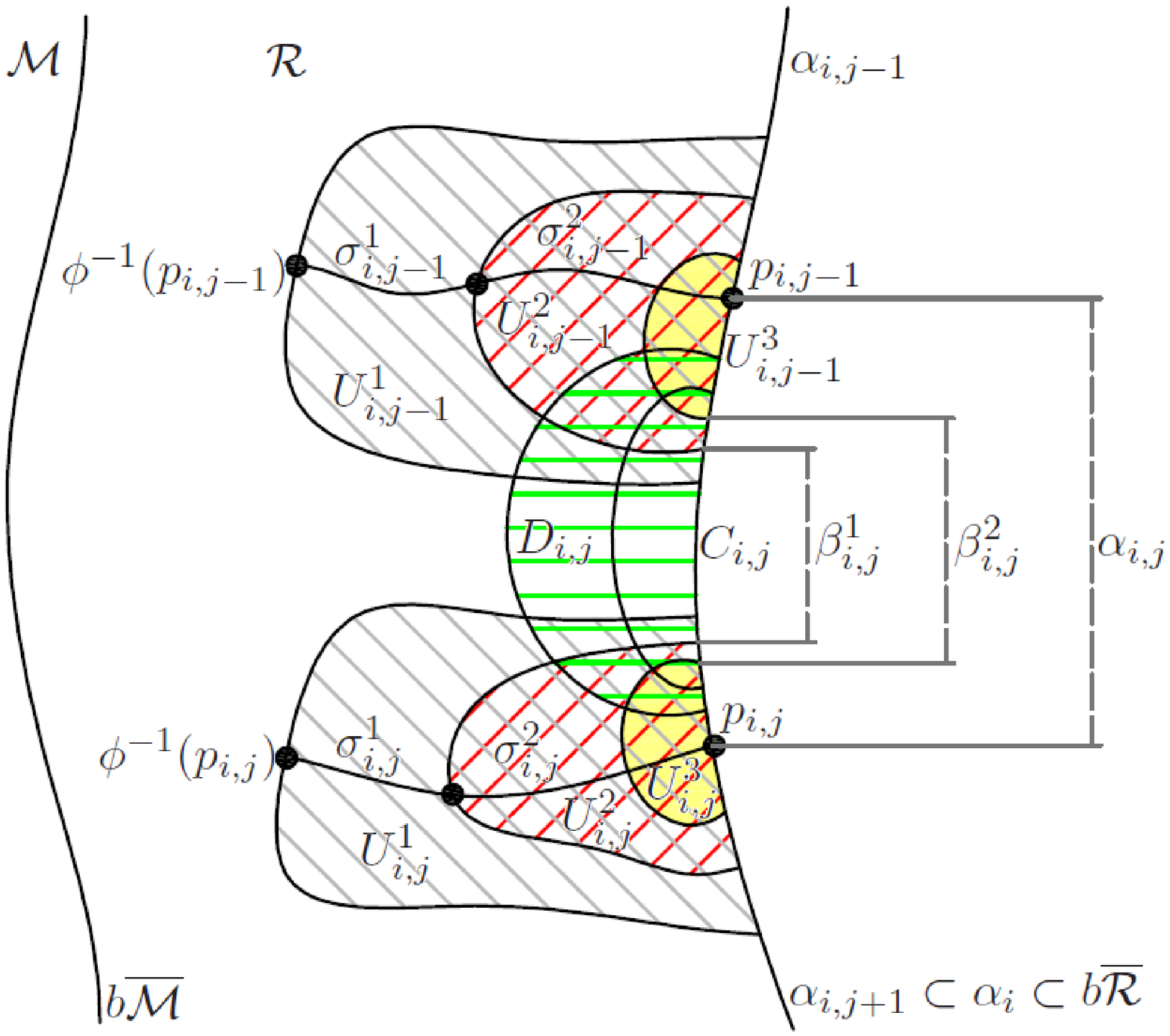}}
        \end{center}
        \vspace{-0.25cm}
\caption{The sets in $\overline{\Rcal}\setminus \overline{\Mcal}$.}\label{fig:R}
\end{figure}


\subsection{Stretching from the arcs $\alpha_{i,j}$}\label{sec:step3}

We shall now stretch the images of the central parts of the arcs $\alpha_{i,j}$ (away from the two endpoints) close to the sphere of radius $\hat s$ in $\c^2$ in order to fulfill the condition {\sf (L2)}. 

For each $(i,j)\in\{1,\ldots,\igot\}\times\z_\jgot$ we choose a smoothly bounded closed disc $\overline{D}_{i,j} \subset\overline{\Rcal}\setminus \overline{\Mcal}$ such that the following hold (see Fig.\ \ref{fig:R}):
\begin{enumerate}[\sf ({e}1)]
\item $\overline{D}_{i,j}\cap\overline{D}_{i,k}=\emptyset$ if $j\neq k$,
\item $\overline{D}_{i,j}\cap \alpha_{i,j}$ is a compact connected Jordan arc that contains $\beta_{i,j}^2$ in its relative interior, 
and $\overline{D}_{i,j}\cap \alpha_{i,k}=\emptyset$ for all $k\neq j$,
\item $\overline{D}_{i,j}\cap \sigma_{i,k}=\emptyset$ for all $k\in \z_\jgot$, and
\item $f_0\big(\overline{D_{i,j}\setminus (U_{i,j-1}^2\cup U_{i,j}^2)}\big)\subset \Uscr_{i,j}$. (Here $f_0$ is the map (\ref{eq:f0}) constructed in the previous subsec.) 
\end{enumerate}
The existence of such discs $D_{i,j}$ is trivially implied by properties {\sf (d4)} and {\sf (d8)}. 
 
At this point we use an approximate solution of a Riemann-Hilbert problem. By property {\sf (e2)} and the second part of {\sf (d6)} we can easily find  a continuous map $g_{i,j}\colon b\overline{D}_{i,j}\times\overline{\d}\to\c^2$ satisfying the following require\-ments:
\begin{enumerate}[$i)$]
\item $g_{i,j}(p,\cdot)\in \Ascr(\d)^2$ and $g_{i,j}(p,0)= f_0(p)$ for all $p\in b\overline{D}_{i,j}$,
\item $g_{i,j}(p,\overline{\d})\subset f_0(p)+\langle f(p_{i,j})\rangle^\bot$ for all $p\in b\overline{D}_{i,j}$,
\item $g_{i,j}(p,b\overline{\d})\subset \b(\hat{s})\setminus \overline{\b}(\hat{s}-\hat{\epsilon})$ for all $p\in b\overline{D}_{i,j}\cap \alpha_{i,j}$, and
\item $g_{i,j}(p,\cdotp) \equiv f_0(p)$ is the constant disc for all $p\in \overline{b {D}_{i,j}\setminus \beta_{i,j}^2}$.
\end{enumerate}

By using a conformal diffeomorphism $\overline{D}_{i,j} \stackrel{\cong}{\longrightarrow} \d$ onto the unit disc, Lemma \ref{lem:Hilbert} gives a map $h_{i,j}\in\Ascr(D_{i,j})^2$ satisfying the following conditions:
\begin{enumerate}[\sf ({f}1)]
\item $h_{i,j}(p)\in \b(\hat{s})\setminus \overline{\b}(\hat{s}-\hat{\epsilon})$ for all $p\in \overline{D}_{i,j}\cap \alpha_{i,j}= b\overline{D}_{i,j}\cap \alpha_{i,j}$; see $iii)$,
\item $\pi_{i,j}\circ h_{i,j}$ is close to $\pi_{i,j}\circ f_0$ on $\overline{D}_{i,j}$; see $ii)$, and
\item $h_{i,j}$ is close to $f_0$ outside a small open neighborhood $C_{i,j}$ of $\beta_{i,j}^2$ in $\overline{D}_{i,j}$ with $b\overline{C}_{i,j} \cap b \overline{D}_{i,j}\subset\alpha_{i,j}$; see $iv)$.
\end{enumerate}
The sets $C_{i,j}\subset D_{i,j}$ and the arcs $\beta^2_{i,j}\subset \alpha_{i,j}$ are illustrated on Fig.\ \ref{fig:R}.

\subsection{Gluing $f_0$ and $h_{i,j}$}\label{sec:step4}
Consider the domains $D_0,D_1$ in $\Rcal$ defined by   
\[
D_0=\Rcal\setminus \cup_{(i,j)\in\{1,\ldots,\igot\}\times\z_\jgot} \overline{C}_{i,j}
\quad\text{and}\quad
D_1=\cup_{(i,j)\in\{1,\ldots,\igot\}\times\z_\jgot} D_{i,j}.
\]
We may assume that the pair $A=\overline D_0$ and $B=\overline D_1$ is a smooth Cartan decomposition of $\overline \Rcal$ (see Def.\ \ref{def:Cartan}). (Smoothness of $bA$ is easily insured by a suitable choice of the neighborhoods $C_{i,j}$ of $\beta_{i,j}^2$, and the separation property follows from the definition of the sets $C_{i,j}$ and $D_{i,j}$.)  

If the approximation of $f_0$ by $h_{i,j}$ on $\overline D_0 \cap \overline{D}_{i,j}$ is close enough (which is guaranteed by {\sf (f3)}), one can apply Corollary \ref{linear-gluing} to obtain a holomorphic map $\hat{f}\colon \overline{\Rcal}\to\c^2$ which is as close as desired to $f_0$ on $\overline D_0$, and to $h_{i,j}$ on $\overline{D}_{i,j}$, for all $(i,j)$. (The cited corollary furnishes a map in $\Ascr(\Rcal)^2$, but we can approximate it by a holomorphic map.)

\begin{remark}
\label{gluing-by-sprays}
In the general case, with $\c^2$ replaced by an arbitrary complex manifold $X$ (see Sec.\ \ref{sec:Stein}), we can instead apply Lemma \ref{lem:Gspray} to obtain a holomorphic map $\hat{f} \colon \overline{\Rcal}\to X$ which is close to the map $f_0$ on $A$, and is close to $h_{i,j}$ on $\overline{C}_{i,j}$ for all $(i,j)$. This can be done provided that we have a suitable procedure for approximating maps $A\to X$ of class $\Ascr(A,X)$, uniformly on $A \cap B$ and uniformly with respect to a parameter $t\in P\subset\c^m$, by maps $B\to X$ of class $\Ascr(B,X)$. We proceed as follows. By Lemma \ref{lem:Espray} we can embed $f_0\colon A \to X$ as the core map $f_0=F_A(0,\cdot)$ in a dominating spray of maps $F_A\colon P \times A \to X$, where $P$ is an open set in $\c^m$ containing the origin. By the approximation procedure we find a spray of maps $F_B \colon P \times B \to X$ such that 
(up to a shrinking of $P$) the map $F_B(t,\cdot)$ is as close as desired to $F_A(t,\cdotp)$ on $A \cap B$ for every $t\in P$. 
Lemma \ref{lem:Gspray} furnishes a new spray of maps $P\times\overline{\Rcal}\to X$ whose core map $\hat{f} \colon \overline{\Rcal}\to X$ is uniformly close to $f_0$ on $A$, and is close to $h_{i,j}$ on $\overline{C}_{i,j}$ for all $(i,j)$.
\end{remark}


\subsection{Checking the properties of $\hat{f}$}\label{sec:step5}

By approximation and general position argument we may assume that $\hat f$ is a holomorphic immersion of a neighborhood of $\overline \Rcal$ in $\widehat \Rcal$ to $\c^2$. Furthermore, if all the approximations in our construction (namely,
\begin{enumerate}[\sf ({A}1)]
\item of $f_0$ by $h_{i,j}$ on $\overline{D}_{i,j}\setminus C_{i,j}$,
\item of $f_0$ by $\hat{f}$ on $\overline D_0$,
\item of $h_{i,j}$ by $\hat{f}$ on $\overline{C}_{i,j}$, and
\item of $\pi_{i,j}\circ f_0$ by $\pi_{i,j}\circ h_{i,j}$ over $\overline{D}_{i,j}$, for all $(i,j)$;
\end{enumerate}
see the above subsections and {\sf (f2)}) are sufficiently close, then $\hat{f}$ satisfies the following properties
which we verify item by item:
\begin{itemize}
\item $\|\hat{f}-f\|_{1,\overline{\Mcal}}<\hat{\epsilon}$.
\end{itemize}

Indeed, just notice that $\|\hat{f}-f_0\|_{1,\overline{\Mcal}}\approx 0$ (see {\sf (A2)} and observe that $\overline{\Mcal}\subset \overline D_0$) and take into account {\sf (c2)}.

\begin{itemize}
\item $\hat{f}(b\overline{\Rcal})\subset \b(\hat{s})\setminus \overline{\b}(\hat{s}-\hat{\epsilon})$. 
\end{itemize}

Indeed, let $p\in b\overline{\Rcal}$. If $p\in C_{i,j}$ for some $(i,j)$, then {\sf (A3)} and {\sf (f1)} give that $\hat{f}(p) \approx h_{i,j}(p)\in \b(\hat{s})\setminus \overline{\b}(\hat{s}-\hat{\epsilon})$. If on the other hand $p\in \alpha_{i,j}\setminus C_{i,j}\subset \overline{U}_{i,j}^3\cup\overline{U}_{i,j+1}^3$ (see Fig.\ \ref{fig:R}), then properties {\sf (A2)} and {\sf (d6)} insure that $\hat{f}(p) \approx f_0(p)\in \b(\hat{s})\setminus \overline{\b}(\hat{s}-\hat{\epsilon})$ as well.

\begin{itemize}
\item $\hat{f}(\overline{\Rcal}\setminus \Mcal)\cap\overline{\b}(s-\epsilon)=\emptyset$. 
\end{itemize}

Indeed, let $p\in \overline{\Rcal}\setminus \Mcal$. If $p\notin \overline{D}_{i,j}$, then {\sf (A2)} and {\sf (c3)} imply that $\hat{f}(p) \approx f_0(p)\notin \overline{\b}(s-\epsilon)$. Assume now that $p\in \overline{D}_{i,j}$. Then either $p\in C_{i,j}$ and {\sf (A3)} gives $\hat{f}(p) \approx h_{i,j}(p)$, or $p\in \overline{D}_{i,j}\setminus C_{i,j}$ and {\sf (A2)} insures that $\hat{f}(p) \approx f_0(p)$. In either case, property {\sf (A4)} guarantees that $\pi_{i,j}(\hat{f}(p))\approx \pi_{i,j}(f_0(p))$ for $p\in \overline{D}_{i,j}$. Assuming that the approximation is close enough (and using compactness of $\overline{\Rcal}\setminus \Mcal$), it thus suffices to verify that $\pi_{i,j}(f_0(p))$ does not belong to the disc $\overline{\d}(s-\epsilon)$ (the $\pi_{i,j}$-projection of the ball $\overline{\b}(s-\epsilon)$). If $p\in \overline{D}_{i,j} \cap\overline{U}_{i,k}^2$, $k\in\{j-1,j\}$, then $\pi_{i,j}(f_0(p)) \in \pi_{i,j}(f_0(\overline{U}_{i,k}^2))$
which is disjoint from $\overline{\d}(s-\epsilon)$ by {\sf (d7)}. Finally, if $p\in \overline{D}_{i,j}\setminus (\overline{U}_{i,j-1}^2\cup\overline{U}_{i,j}^2)$, then {\sf (e4)} gives that $f_0(p)\in \Uscr_{i,j}$ and so $\pi_{i,j}(f_0(p))\notin \overline{\d}(s-\epsilon)$ as well; see (\ref{eq:far}) and {\sf(a3)}. 

\begin{itemize}
\item $\dist_{(\Rcal,\hat{f})}(z_0,b\overline{\Rcal})>\hat{\rho}=\rho+\delta$.
\end{itemize}

By {\sf (A2)} and {\sf (c4)} it suffices to check that $\dist_{(\Rcal,\hat{f})}(b\overline{\Mcal},b\overline{\Rcal})>\delta$. Let $\gamma$ be any curve in $\overline \Rcal\setminus \Mcal$ connecting $b\overline{\Mcal}$ and $b\overline{\Rcal}$, and let us show that $\length(\hat{f}(\gamma))>\delta$. 

Assume first that $\gamma\cap \overline{U}_{i,j}^2\neq\emptyset$ for some $(i,j)$. Then there exists a subarc $\hat{\gamma}\subset \overline{U}_{i,j}^1$ connecting $\Rcal\setminus U_{i,j}^1$ and $\overline{U}_{i,j}^2$; see Fig.\ \ref{fig:R}. By {\sf (A2)} and {\sf (A3)} one has
\begin{eqnarray*}
\length(\hat{f}(\hat{\gamma}))& \approx & \length (f_0(\hat{\gamma}\cap \overline{D}_0))
\\
& & +\ \length (h_{i,j}(\hat{\gamma}\cap \overline{C}_{i,j})) + \length (h_{i,j+1}(\hat{\gamma}\cap \overline{C}_{i,j+1}))\\
& \geq & \length (f_0(\hat{\gamma}\cap \overline{D}_0)) +\ \length (\pi_{i,j}(h_{i,j}(\hat{\gamma}\cap \overline{C}_{i,j}))) \\
 & & +\ \length (\pi_{i,j+1}(h_{i,j+1}(\hat{\gamma}\cap \overline{C}_{i,j+1}))),
\end{eqnarray*} 
and by {\sf (A4)} and {\sf (d9)}, the above is
\begin{eqnarray*}
& \approx & \length (f_0(\hat{\gamma}\cap \overline{D}_0))
\\
& & +\ \length (\pi_{i,j}(f_0(\hat{\gamma}\cap \overline{C}_{i,j}))) + \length (\pi_{i,j+1}(f_0(\hat{\gamma}\cap \overline{C}_{i,j+1}))) > \delta.
\end{eqnarray*}
Therefore, $\length(\hat{f}(\gamma)) \ge \length(\hat{f}(\hat{\gamma}))>\delta$ as claimed.

Assume now that $\gamma\cap\overline{U}_{i,j}^2=\emptyset$ for all $(i,j)\in\{1,\ldots,\igot\}\times\z_\jgot$. Then there exist $(i,j)$ and a subarc $\hat{\gamma}\subset \overline{D_{i,j}\setminus (U_{i,j-1}^2\cup U_{i,j}^2)}$ connecting a point $p_0\in \overline{D_{i,j}\setminus (U_{i,j-1}^2\cup U_{i,j}^2\cup C_{i,j})}$ to a point $p_1\in \beta_{i,j}^1$; see Fig.\ \ref{fig:R}. In this case, {\sf (A2)} and {\sf (A3)} imply that $\hat{f}(p_0)\approx f_0(p_0)$ and $\hat{f}(p_1)\approx h_{i,j}(p_1)$. By {\sf (e4)} one has $\{f_0(p_0),f_0(p_1)\}\subset\Uscr_{i,j}$, by {\sf (f1)} one has  $h_{i,j}(p_1)\in \b(\hat{s})\setminus \overline{\b}(\hat{s}-\hat{\epsilon})$, and by {\sf (A4)} one has $\pi_{i,j}(h_{i,j}(p_1))\approx \pi_{i,j}(f_0(p_1))$; hence inequality (\ref{eq:Pytha}) gives that $\dist(f_0(p_0),h_{i,j}(p_1))>\delta$. Assuming as we may that $\hat{f}(p_0)$ and $\hat{f}(p_1)$ are close enough to $f_0(p_0)$ and $h_{i,j}(p_1)$, respectively, we thus get $\length(\hat{f}(\gamma)) \ge \length(\hat{f}(\hat{\gamma}))>\delta$.

Thus $\hat{f}$ satisfies the conclusion of Lemma \ref{lem:main} and the proof is complete.


\section{Complete proper immersions to Euclidean balls} \label{sec:proof}

We now show how Lemma \ref{lem:main} implies Theorem \ref{th:intro2}.

Let $\overline{\Rcal}$ be a bordered Riemann surface and $h \colon \overline{\Rcal}\to\c^m$ $(m>1)$ be a holomorphic map whose image $h(\overline{\Rcal})$ is contained in an open ball $B \subset \c^m$. Given a compact set $K\subset \Rcal$ and a number $\eta>0$, we must find a complete proper holomorphic immersion $\hat h\colon \Rcal \to B$ (embedding if $m>2$) such that $\|\hat{h}-h\|_{1,K}<\eta$.

We focus on immersions; the necessary modifications to find embeddings in $\c^m$ for $m>2$ are indicated at the end.

By a translation and a dilation of coordinates we may assume that $B=\b$ is the unit ball. By general position we can replace $h$ by an immersion (embedding if $m>2$) such that $h(b\overline{\Rcal})$ does not contain the origin of $\c^m$. Pick numbers $0< \xi < r<1$ so that $h(b\overline{\Rcal}) \subset \b(r)\setminus \overline{\b}(r-\xi)$. Choose a bordered domain $\Ncal\in\Bscr(\Rcal)$ in $\Rcal$ such that  
\[
	K\subset \Ncal \quad \hbox{and} \quad 
	h(\overline{\Rcal} \setminus  \Ncal)\subset \b(r)\setminus \overline{\b}(r-\xi).
\]
Set $c :=\sqrt{6(1-r^2)}/\pi>0$. Fix a point $\zeta_0\in \Ncal$ and define sequences 
$\rho_n,r_n>0$ $(n\in \z_+)$ recursively as follows:  
\begin{equation}\label{eq:r-rho}
\rho_0=\dist_{(\Rcal,h)}(\zeta_0,b\overline{\Ncal}),\enskip
\rho_n=\rho_{n-1}+\frac{c}{n},\enskip r_0=r, \enskip r_n=\sqrt{r_{n-1}^2+\frac{c^2}{n^2}}.
\end{equation}
It is immediate that
\begin{equation}\label{eq:diverges}
\lim_{n\to\infty}\rho_n=+\infty
\end{equation}
and
\begin{equation}\label{eq:converges}
\lim_{n\to\infty}r_n=1.
\end{equation}
To get (\ref{eq:converges}), observe that $r_n^2=r_{n-1}^2+\frac{c^2}{n^2}$ and hence
\[
	\lim_{n\to\infty}r_n^2 = r^2 + \sum_{n=1}^\infty \frac{c^2}{n^2} = 
	r^2+ c^2 \frac{\,\,\pi^2}{6} = r^2 + \frac{6(1-r^2)}{\pi^2} \frac{\,\,\pi^2}{6} =  1.
\]	 

Set $(\Ncal_0,h_0,\xi_0)=(\Ncal,h,\xi)$. We shall inductively construct sequences of bordered domains $\Ncal_n\in\Bscr(\Rcal)$, holomorphic immersions $h_n\in\Iscr(\Rcal)$, and constants $\xi_n>0$ $(n\in\n)$ satisfying the following conditions:
\begin{enumerate}[\sf (a$_n$)]
\item $\Ncal_{n-1}\Subset \Ncal_n$, 
\item $\displaystyle \xi_n <\min\left\{ \xi_{n-1}\,,\, \frac{\eta}{2^{n}} \;,\; \frac{1}{2^{n}}\min\left\{ \min_{\Ncal_k} |dh_k| \colon k=1,\ldots,n-1 \right\}\right\}$; this minimum is positive since $dh_k\ne 0$ on $\overline \Rcal$ for all $k\in\{1,\ldots,n-1\}$,
\item $\|h_n-h_{n-1}\|_{1,\overline{\Ncal}_{n-1}}<\xi_n$,
\item $h_n(\overline{\Rcal}\setminus \Ncal_n)\subset \b(r_n)\setminus \overline{\b}(r_n-\xi_n)$,
\item $h_n(\overline{\Rcal}\setminus \Ncal_{n-1})\cap \overline{\b}(r_{n-1}-\xi_{n-1})=\emptyset$, and
\item $\dist_{(\Rcal,h_n)}(\zeta_0,b\overline{\Ncal}_n)>\rho_n$.
\end{enumerate}
We will additionally insure that 
\begin{equation}\label{eq:cupN}
\Rcal=\cup_{n\in\n}\Ncal_n.
\end{equation}

We proceed by induction. To begin, pick a number $\xi_1>0$ so that {\sf (b$_1$)} holds. Now Lemma \ref{lem:main} can be applied to the data
\[
(\Mcal \,,\, z_0 \,,\, f \,,\, \epsilon \,,\, \rho \,,\, s \,,\, \hat{\epsilon} \,,\, \delta)= (\Ncal_0 \,,\, \zeta_0 \,,\, h_0 \,,\, \xi_0 \,,\, \rho_0 \,,\,r_0 \,,\, \xi_1 \,,\,c),
\]
and define $h_1\in\Iscr(\Rcal)$ as the immersion $\hat{f}$ furnished by Lemma \ref{lem:main}. Properties {\sf (c$_1$)}, {\sf (e$_1$)} correspond to {\sf (L1)}, {\sf (L3)} in Lemma \ref{lem:main}, respectively; take into account (\ref{eq:r-rho}). Moreover, {\sf (L2)} and {\sf (L4)} give that $h_1(b\overline{\Rcal})\subset \b(r_1)\setminus \overline{\b}(r_1-\xi_1)$ and $\dist_{(\Rcal,h_1)}(\zeta_0,b\overline{\Rcal})>\rho_1$ (see (\ref{eq:r-rho}) again). Since $b\overline{\Rcal}$ is compact, it follows that {\sf (d$_1$)} holds for any sufficiently large domain $\Ncal_1\in\Bscr(\Rcal)$ satisfying {\sf (a$_1$)}.

For the inductive step, assume that for some $n\ge 2$ we already have $(\Ncal_{j},h_{j},\xi_{j})$ satisfying properties {\sf (a$_j$)}--{\sf (f$_j$)} for all $j\in\{1,\ldots,n-1\}$. Choose any positive number $\xi_n>0$ satisfying {\sf (b$_n$)}. Property {\sf (d$_{n-1}$)} allows to apply Lemma \ref{lem:main} to the data
\[
(\Mcal \,,\, z_0 \,,\, f \,,\, \epsilon \,,\, \rho \,,\, s \,,\, \hat{\epsilon} \,,\, \delta)= (\Ncal_{n-1} \,,\, \zeta_0 \,,\, h_{n-1} \,,\, \xi_{n-1} \,,\, \rho_{n-1} \,,\,r_{n-1} \,,\, \xi_n \,,\,\frac{c}{n}).
\]
Let $h_n$ denote the immersion $\hat f$ furnished by Lemma \ref{lem:main}. By choosing a sufficiently large domain $\Ncal_n \in \Bscr(\Rcal)$, the triple $(\Ncal_n,h_n,\xi_n)$ meets all requirements {\sf (a$_n$)}--{\sf (f$_n$)}. 

Finally, in order to guarantee (\ref{eq:cupN}), one simply chooses the bordered domain $\Ncal_n$ large enough in each step of the inductive process. This concludes the construction of the sequence $\{(\Ncal_n,h_n,\xi_n)\}_{n\in\n}$.

From {\sf (c$_n$)}, {\sf (b$_n$)}, and (\ref{eq:cupN}) we infer that the sequence $\{h_n\colon \overline {\Rcal}\to\c^2\}_{n\in\n}$ converges uniformly on compacta in $\Rcal$ to a holomorphic map $\hat{h} \colon \Rcal\to\c^2$. Let us show that $\hat{h}$ satisfies the desired properties.

\begin{itemize}
\item $\hat{h}\colon \Rcal\to\c^2$ is an immersion.
\end{itemize}

Indeed, let $p\in \Rcal$. Pick $n_0\in\n$ so that $p\in\Ncal_{n_0}$. From {\sf (c$_n$)} and {\sf (b$_n$)} one has
\begin{eqnarray*}
|d\hat{h}(p)| & \geq & |dh_{n_0}(p)|-\sum_{n>n_0} |dh_{n}(p)-dh_{n-1}(p)|
\\
  & \geq & |dh_{n_0}(p)|-\sum_{n>n_0} \|h_{n}-h_{n-1}\|_{1,\Ncal_{n_0}}
\\
  & \geq & |dh_{n_0}(p)|-\sum_{n>n_0} \xi_n
\\
  & \geq & |dh_{n_0}(p)|-\sum_{n>n_0} \frac{1}{2^n} |dh_{n_0}(p)|> \frac12 \, |dh_{n_0}(p)|>0;
\end{eqnarray*}
recall that $h_{n_0}\in\Iscr(\Rcal)$. This shows that $\hat{h}$ is an immersion as claimed.

\begin{itemize}
\item $\hat{h}\colon \Rcal\to\c^2$ is complete.
\end{itemize}

Indeed, arguing as above, one infers that
\[
\dist_{(\Rcal,\hat{h})}(\zeta_0,b\overline{\Ncal}_n)>\frac12\,\dist_{(\Rcal,h_n)}(\zeta_0,b\overline{\Ncal}_n)> \frac{\rho_n}2\quad\forall n\in\n;
\]
see {\sf (f$_n$)}. Taking limits in the above inequality as $n$ goes to infinity, one gets the completeness of $\hat{h}$ from (\ref{eq:diverges}).

\begin{itemize}
\item $\|\hat{h}-h\|_{1,\overline{\Ncal}}<\eta$. 
\end{itemize}

This follows trivially from {\sf (b$_n$)} and {\sf (c$_n$)}.

\begin{itemize}
\item $\hat{h}(\Rcal)\subset\b$ and $\hat{h}\colon \Rcal\to\b$ is proper. 
\end{itemize}

Indeed, let $p\in\Rcal$. From {\sf (d$_n$)} and the Maximum Principle, $|h_n(p)|<r_n$ for every $n\in\n$. By (\ref{eq:converges}), taking limits as $n$ goes to infinity, one has  $|\hat{h}(p)|\leq1$ and, again by the Maximum Principle, the first assertion holds.

For the properness it suffices that $\hat{h}^{-1}(\overline{\b}(t))$ is a compact subset of $\Rcal$ for any $0<t<1$. Observe first that {\sf (b$_n$)} and {\sf (c$_n$)} imply  
\begin{equation}\label{eq:hath-h}
\|\hat{h}-h_n\|_{0,\Ncal_n}<\eta/2^n\quad \forall n\in\n.
\end{equation} 
Since $r_n\to 1$ and $\xi_n\to 0$ as $n\to\infty$, we can take $n_0\in\n$ large enough so that 
\begin{equation}\label{eq:r-t}
t+\xi_{n-1}+\eta/2^n<r_{n-1}\quad \forall n\geq n_0.
\end{equation}
Combining {\sf (e$_n$)} and (\ref{eq:hath-h}) one infers that
\[
\hat{h}(\overline{\Ncal}_n\setminus \Ncal_{n-1})\cap \overline{\b}(r_{n-1}-\xi_{n-1}-\eta/2^n)=\emptyset\quad \forall n\geq n_0.
\]
Therefore, (\ref{eq:r-t}) gives that 
\[
(\overline{\Ncal}_n\setminus \Ncal_{n-1})\cap \hat{h}^{-1}(\overline{\b}(t))=\emptyset\quad \forall n\geq n_0,
\]
hence $\hat{h}^{-1}(\overline{\b}(t))\subset \Ncal_{n_0}$ is compact in $\Rcal$ and we are done.

This completes the proof of Theorem \ref{th:intro2} for the case of immersions. If $m>2$, we construct the sequence $h_n\colon \overline \Rcal \to \b\subset \c^m$ so that, in addition, $h_n$ is an embedding on the bordered domain $\overline \Ncal_n \subset \Rcal$ for every $n$; this is possible by applying the general position argument at each step. Furthermore, if the approximation of $h_n$ by $h_{n+1}$ is sufficiently close on $\overline \Ncal_n$ (which is insured by choosing the sequence $\xi_n>0$ to converge to zero fast enough), then the limit map $\hat h=\lim_{n\to\infty} h_n$ is also an embedding. This can be done, for instance, taking in addition to {\sf (b$_n$)},
\[
   \xi_n<\frac1{2n^2} \inf\left\{ |h_{n-1}(p)-h_{n-1}(q)|\colon p,q\in \Ncal_{n-1},\, 
   {\sf d}(p,q)>\frac1{n}\right\},
\]
where ${\sf d}(\cdot,\cdot)$ is any fixed Riemannian metric in $\Rcal$; see the proof of Theorem 4.5 in \cite{AL3} or \cite{DF}.


\section{Complete proper immersions to Stein manifolds} \label{sec:Stein}

By combining our methods with those in \cite{DF} we also find complete proper holomorphic immersions of any bordered Riemann surface to an arbitrary Stein manifold of dimension $>1$.

\begin{theorem}
\label{th:Stein}
Let $X$ be a Stein manifold of dimension $>1$ endowed with a hermitian metric $ds^2_X$, and let $\overline \Rcal$ be a bordered Riemann surface. Every holomorphic map $f\colon \overline \Rcal\to X$ can be approximated, uniformly on compacta in $\Rcal$, by proper complete holomorphic immersions $\hat f\colon \Rcal\to X$. If $\dim X\ge 3$ then $\hat f$ can be chosen an embedding. 
\end{theorem}

\begin{proof} 
We outline the necessary modifications to the proof of Theorem \ref{th:intro2}.

Consider first the case when $X$ is a relatively compact, smoothly bounded, strongly pseudoconvex domain in another Stein manifold $Y$, and the hermitian metric $ds^2_X$ extends to a metric $ds^2_Y$ on $Y$. Choose a smooth strongly pluri\-sub\-harmonic function $\rho$ in a neighborhood of $\overline X$ in $Y$ such that $X=\{\rho<0\}$ and $d\rho\ne 0$ on $bX=\{\rho<0\}$. Pick a number $c_0>0$ such that $\rho$ has no critical values in $[-c_0,+c_0]$. 

Every point $p\in bX$ admits a pair of open neighborhoods $U'_p\Subset U_p\Subset Y$ and a holomorphic coordinate map $\theta_p\colon U_p \stackrel{\cong}{\longrightarrow} \b\subset \c^n$ onto a ball $\b$ in $\c^n$, centered at $0=\theta_p(p)$, such that $\theta(U'_p)$ is a smaller ball $\b' \subset \b$ also centered at $0$, and the function $\rho_p := \rho\circ \theta_p^{-1}\colon \b \to\r$ is strongly convex. In particular, $\Sigma_p := \theta_p(bX\cap U_p)$ is a strongly convex hypersurface in $\b$ that may be chosen $\Ccal^2$-close to a spherical cap. 
Furthermore, by shrinking $U_p$ we insure that the metric $ds^2_Y$ pulls back by $\theta_p$ to a hermitian metric on $\b$ that is comparable to the standard Euclidean metric. By choosing the neighborhood $U'_p$ small enough compared to $U_p$ we can also arrange that, for any point $q\in U'_p$, the intersection of the tangent plane at the point $q':=\theta_p(q) \in\b'$ to the strongly convex hypersurface $\{\rho_p=\rho_p(q')\} \subset \b$ with the convex domain $D_p:=\theta_p(D\cap U_p) \subset\b$ is a relatively compact subset of $\b$.

By compactness of $bX$ there are finitely many holomorphic coordinate charts $\Ucal=\{(U_j,\theta_j)\}_{j=1,\ldots,m}$ as above, and the corresponding subsets $U'_j \Subset U_j$, such that $bX \subset U':=\cup_{j=1}^m U'_j$. 

By decreasing the constant $c_0>0$ if necessary we can insure that $U'$ contains the collar $\{z \in Y \colon -c_0\le \rho(z) \le +c_0\}$. 

Fix a map $f\colon \overline R\to X$ of class $\Ascr(\Rcal,X)$. Pick a constant $c_1>0$ such that $f(\overline \Rcal) \subset \{z \in X\colon \rho(z)<-c_1\}$. Now choose a constant $c$ with $0<c<\min\{c_0,c_1\}$, and a number $\epsilon$ with $0<\epsilon <c$. By \cite{DF} we can approximate the map $f$, uniformly on a given compact set in $\Rcal$, by a map $f_0\in \Ascr(\Rcal,X)$ satisfying $-c<\rho(f_0(x)) < -c+\epsilon$ for all $x\in b\Rcal$. The proof in \cite{DF} uses the tools described in Sec.\ \ref{sec:preli} above; in particular, the Riemann-Hilbert problem and the method of gluing sprays. By general position we may assume that $f_0$ is an immersion. Now replace $f$ by $f_0$ and assume that $f$ satisfies these properties.

We now follow the construction in Sec.\ \ref{sec:mainlemma} to find a new immersion $\hat f\colon \overline \Rcal \to X$ which satisfies an analogue of Lemma \ref{lem:main} in this setting. We begin by subdividing the boundary $b\Rcal$ into subarcs $\alpha_{i,j}$ as in Subsec.\ \ref{sec:step1} such that each arc $f(\alpha_{i,j})\subset D$ is contained in one of the sets $U'_k$, and it satisfies the relevant conditions stated in Subsec.\ \ref{sec:step1} with respect to the local holomorphic coordinates on $U'_k$. We then perform the same construction as in the proof of Lemma \ref{lem:main} within the local chart $(U_k,\theta_k)$. The geometric conditions described above enable the use of stretching, first from the two endpoints of $\alpha_{i,j}$, and then from the middle segment, towards the boundary $bX\cap U_k$ so that the induced boundary distance in $\Rcal$ increases by a specific amount. Since the geometry in each local chart is essentially the same as in the model case (
 recall that $\theta_k(bX\cap U_k)$ is close to a spherical cap and the hermitian metric from $X$ is comparable with the Euclidean metric), the relevant estimates in Sec.\ \ref{sec:mainlemma} remain valid up to a uniformly bounded numerical factor whose presence does not present any problem in the proof. Each of the local modifications obtained in this way is glued with the existing immersion on the rest of the domain by the method of gluing sprays; see Lemma \ref{lem:Gspray} and Remark \ref{gluing-by-sprays}. 

In this way we approximate $f$, uniformly on a given compact in $\Rcal$, by a new immersion $\hat f\colon \overline \Rcal \to X$ so that the boundary moves closer to $bX$ by a controlled amount, and the boundary distance in the immersed curve also increases by a presribed amount, the two numbers being related in the Pythagoras' way; compare with Lemma \ref{lem:main}. The details are very similar to those given above and will be left out. The proof is finished by an induction just as in the proof of Theorem \ref{th:intro2}. 

This settles the case when $X$ is a bounded strongly pseudoconvex domain.

The general case is obtained as follows. Choose a smooth strongly plurisubharmonic exhaustion function $\rho\colon X\to\r$ with Morse critical po\-ints. Let $f\colon \overline \Rcal\to X$ be a map of class $\Ascr(\Rcal,X)$. Pick an increasing sequence $c_0<c_1<c_2\cdots$, with $\lim_{j\to\infty} c_j=+\infty$, such that every $c_j$ is  regular value of $\rho$ and $f(\overline \Rcal)\subset \{\rho<c_0\}$. The sets $X_j=\{\rho<c_j\}$ for $j=0,1,\ldots$ are smoothly bounded strongly pseudoconvex domains exhausting $X$. Fix a point $p_0\in \Rcal$. Choose an increasing sequence $0< M_0< M_1 < M_2<\cdots$ with $\lim M_j=+\infty$. By the special case explained above we can approximate $f$, uniformly on a given compact set $K\subset \Rcal$, by a holomorpic immersion $f_0\colon \overline \Rcal \to X_0$ such that $f_0(b\Rcal)$ is very close to $bX_0$, and for any curve $\gamma\subset \Rcal$ connecting $p_0$ to the boundary $b\Rcal$, the image curve $f_0(\gamma) \subset X_0$ has length a
 t  least $M_0+1$ with respect to $ds^2_X$. We can find a relatively compact subdomain $\Rcal_0\Subset \Rcal$ very close to $\Rcal$ such that the distance from $p_0$ to $\Rcal\setminus \Rcal_0$ with respect to the metric $f_0^* (ds^2_X)$ is $>M_0$, and such that $f_0(\overline \Rcal \setminus \Rcal_0)$ lies in a small neighborhood of $bX_0$. Applying again the special case we approximate the map $f_0$, uniformly on $\overline \Rcal_0$, by a holomorphic immersion $f_1\colon \overline \Rcal \to X_1$ such that $f_1(b\Rcal)$ is contained in a small neighborhood of $bX_1$ and the distance from $p_0$ to $b\Rcal$ with respect to the metric $f_1^* (ds^2_X)$ is $>M_1+1$. Next choose a domain $\Rcal_1\Subset \Rcal$ containing $\overline \Rcal_0$ such that the distance from $p_0$ to $\Rcal\setminus \Rcal_1$ in the metric $f_1^* (ds^2_X)$ is $>M_1$ and $f_1(\overline \Rcal\setminus \Rcal_1)$ is very close to $bX_1$. 

Continuing inductively, we obtain a sequence of holomorphic immersions $f_j\colon \overline\Rcal \to X_j$ that converges uniformly on compacta in $\Rcal$ to a proper complete holomorphic immersion $f=\lim_{j\to\infty} f_j\colon \Rcal\to X$. When $\dim X\ge 3$, the maps $f_j$ in the above sequence, and also the limit map $f$, can be chosen to be embeddings.
\end{proof}

\noindent\bf Acknowledgements. \rm
A.\ Alarc\'{o}n is supported by Vicerrectorado de Pol\'{i}tica Cient\'{i}fica e Investigaci\'{o}n de la Universidad de Granada, and is partially supported by MCYT-FEDER grants MTM2007-61775 and MTM2011-22547, Junta de Andaluc\'{i}a Grant P09-FQM-5088, and the grant PYR-2012-3 CEI BioTIC GENIL (CEB09-0010) of the MICINN CEI Program. 

F.\ Forstneri\v c is supported by the research program P1-0291 from ARRS, Republic of Slovenia.

\end{document}